\documentclass[12pt]{amsart}

\usepackage{amssymb}
\usepackage{epsfig}
\usepackage{url}
\usepackage{array}
\usepackage{varwidth}
\usepackage{slashbox}

\newtheorem{theorem}{Theorem}[section]
\newtheorem{lemma}[theorem]{Lemma}

\theoremstyle{definition}

\newtheorem{conjecture}[theorem]{Conjecture}
\numberwithin{equation}{section}

\theoremstyle{remark}

\numberwithin{equation}{section}

\newcommand{\I}{{\mathcal{I}}}
\newcommand{\B}{{\mathcal{B}}}

\title[]{$h$-vectors of Small Matroid Complexes}

\author{Jes\'us A. De Loera, Yvonne Kemper, Steven Klee}

\date{\today}

\begin{document}

\maketitle


\begin{abstract}
Stanley conjectured in 1977 that the $h$-vector of a matroid simplicial complex is a pure $O$-sequence.  We give simple constructive proofs that the conjecture is true for matroids of rank less than or equal to 3, and corank 2. We used computers to verify that Stanley's conjecture holds for all matroids on at most nine elements.
\end{abstract}

\section{Introduction}

Before stating the key goal of our investigations and stating our results, we will briefly review some relevant background material on matroids and simplicial complexes.  For further information, we refer the reader to the books of Oxley \cite{Oxley-book}, White \cite{White-book}, and Stanley \cite{Stanley-book}.

Recall that a matroid $M = (E(M),\I(M))$ consists of a \textit{ground set} $E(M)$ and a family of subsets $\mathcal{I}(M) \subseteq 2^{E(M)}$ called \textit{independent sets} such that
\begin{enumerate}
\item[(I1):] $\emptyset \in \I(M)$;
\item[(I2):] If $A \in \I(M)$ and $A' \subset A$, then $A' \in \I(M)$; and
\item[(I3):] If $A,A' \in \I(M)$ with $|A| < |A'|$, then there is some $e \in A'\setminus A$ such that $A \cup e \in \I(M)$.
\end{enumerate}

Equivalently, the independent sets of a matroid $M$ on the ground set
$E(M)$ form a simplicial complex, called the \textit{independence complex of $M$}, with the property that the
restriction $\mathcal{I}(M)|_{E'}$ is pure for any subset $E'
\subseteq E(M)$.  A \textit{basis} of $M$ is a maximal independent set
under inclusion.  The \textit{rank} of a subset $E' \subseteq E(M)$ is the size of the largest independent set $A \subseteq E'$; in particular, the \textit{rank} of $M$ is the cardinality of a basis.  A \textit{loop} is a singleton $\{e\} \notin \mathcal{I}(M)$.  Since the loops of a matroid are not seen by the independence complex, no generality will lost in only considering loopless matroids.  

If $M$ is a loopless matroid, elements $e,e' \in E(M)$ are \textit{parallel} if $\{e,e'\} \notin \mathcal{I}(M)$.  The \textit{parallelism classes} of $M$ are maximal subsets $E_1,\ldots,E_t \subseteq E(M)$ with the property that all elements in each set $E_i$ are parallel.  It can be easily checked that if $\{e_{i_1},\ldots,e_{i_k}\} \in \mathcal{I}(M)$ with $e_{i_j} \in E_{i_j}$, then $\{e'_{i_1},\ldots,e'_{i_k}\} \in \mathcal{I}(M)$ for any choice of $e'_{i_j} \in E_{i_j}$.  Alternatively, the parallelism classes of $M$ are maximal rank-one subsets of $E(M)$. 

Given a matroid $M$ on the ground set $E(M)$ with bases $\B(M)$, we define its \textit{dual matroid}, $M^*$, to be the matroid on $E(M)$ whose bases are $\B(M^*) = \{E\setminus B: B \in \B(M)\}$.  We say that $\{e\}$ is a \textit{coloop} in $M$ if $\{e\}$ is a loop in $M^*$ or, equivalently, if each basis of $M$ contains $e$.  

If $M$ is a matroid of rank $d$, the \textit{$f$-vector} of $M$ is $f(M):=(f_{-1}(M),f_0(M),\ldots,f_{d-1}(M))$, whose entries are $f_{i-1}(M):=|\{A \in \I(M): |A| = i\}|$.  Oftentimes, it is more convenient to study the \textit{$h$-vector} $h(M):=(h_0(M),\ldots,h_d(M))$ whose entries are defined by the relation $$\sum_{j=0}^dh_j(M)\lambda^{j} = \sum_{i=0}^df_{i-1}(M)\lambda^i(1-\lambda)^{d-i}.$$   See \cite{Stanley-book} for more on $h$-vectors and the combinatorics of simplicial complexes.

It should not be expected that the $h$-numbers of a general simplicial complex are nonnegative; however, the $h$-numbers of a matroid $M$ may be interpreted combinatorially in terms of certain invariants of $M$. Fix a total ordering $\{v_1<v_2<\ldots<v_n\}$ on $E(M)$.  Given a basis $B \in \I(M)$, an element $v_j \in B$ is \textit{internally  passive in $B$} if there is some $v_i \in E(M) \setminus B$ such that $v_i<v_j$ and $(B\setminus v_j)\cup v_i$ is a basis of $M$.  Dually, $v_j \in E(M)\setminus B$ is \textit{externally passive in $B$} if there is an element $v_i \in B$ such that $v_i<v_j$ and $(B \setminus v_i) \cup v_j$ is a basis.  (Alternatively, $v_j$ is externally passive in $B$ if it is internally passive in $E(M) \setminus B$ in $M^*$.)  It is well known (\cite[Equation (7.12)]{White-book}) that 
\begin{equation} \label{h-nums-matroid} 
\sum_{j=0}^dh_j(M)\lambda^j = \sum_{B \in \B(M)}\lambda^{ip(B)},
\end{equation}
where $ip(B)$ counts the number of internally passive elements in $B$.  This proves that the $h$-numbers of a matroid complex are nonnegative.   Alternatively,
\begin{equation} \label{h-nums-dual}
\sum_{j=0}^dh_j(M)\lambda^j = \sum_{B \in \B(M^*)}\lambda^{ep(B)},
\end{equation}
where $ep(B)$ counts the number of externally passive elements in $B$.  Since the $f$-numbers (and hence the $h$-numbers) of a matroid depend only on its independent sets, equations \eqref{h-nums-matroid} and \eqref{h-nums-dual} hold for \textit{any} ordering of the ground set of $M$. It is worth
remarking that the $h$-polynomial above is actually a specialization of the well-known \emph{Tutte polynomial} of the corresponding matroid 
(see \cite{White-book}).

An \textit{order ideal} $\mathcal{O}$ is a family of monomials (say of degree at most $r$) with the property that if $\mu \in \mathcal{O}$ and $\nu|\mu$, then $\nu \in \mathcal{O}$.  Let $\mathcal{O}_i$ denote the collection of monomials in $\mathcal{O}$ of degree $i$.  Let $F_i(\mathcal{O}):=|\mathcal{O}_i|$ and  $F(\mathcal{O}) = (F_0(\mathcal{O}),F_1(\mathcal{O},\ldots,F_r(\mathcal{O}))$.  We say that $\mathcal{O}$ is \textit{pure} if all of its maximal monomials (under divisibility) have the same degree.  A vector $\mathbf{h} = (h_0,\ldots,h_d)$ is a \textit{pure $O$-sequence} if there is a pure order ideal $\mathcal{O}$ such that $\mathbf{h} = F(\mathcal{O})$.

A longstanding conjecture of Stanley \cite{Stanley-CM-complexes} suggests that matroid $h$-vectors are highly structured.
\begin{conjecture} \label{Stanley-conjecture}
For any matroid $M$, $h(M)$ is a pure $O$-sequence.
\end{conjecture}

Conjecture \ref{Stanley-conjecture} is known to hold for several
families of matroid complexes, such as paving matroids
\cite{Merino-paving}, cographic matroids \cite{Merino-cographic},
cotransversal matroids \cite{Oh-contransversal}, lattice path matroids
\cite{Schweig-LP}, and matroids of rank at most three
\cite{Ha-Stokes-Zanello, Stokes-thesis}.  The purpose of this paper is
to present a proof of Stanley's conjecture for all matroids on at
most nine elements, all matroids of corank two and all matroids of
rank at most three.  While Stanley's conjecture was known to hold for
matroids of rank two \cite{Stokes-thesis} and rank three
\cite{Ha-Stokes-Zanello}, we use the geometry of the independence
complexes of matroids of small rank to provide much simpler shorter proofs in
these cases.  Our results show that any counterexample to Stanley's
conjecture must have at least ten elements, rank at least four, and corank at least three.

This article will use several ideas from the theory of multicomplexes
and monomial ideals. Although a general classification of matroid
$h$-vectors or pure $O$-sequences seems to be an incredibly difficult
problem, some properties are known and will be used in the proofs
below:

\begin{theorem}\cite{Brown-Colbourn, Chari, Hibi}
Let $\mathbf{h} = (h_0,h_1,\ldots,h_d)$ be a matroid $h$-vector or a pure $O$-sequence with $h_d \neq 0$.  Then
\begin{enumerate}
\item $h_0 \leq h_1 \leq \cdots \leq h_{\lfloor \frac{d}{2}\rfloor}$,
\item $h_i \leq h_{d-i}$ for all $0 \leq i \leq \lfloor \frac{d}{2} \rfloor$, and
\item for all $0 \leq s \leq d$ and $\alpha \geq 1$, we have
\begin{equation}\label{BC-ineq}
\sum_{i=0}^s(-\alpha)^{s-i}h_i \geq 0.
\end{equation}
\end{enumerate}
\end{theorem}
Inequality \eqref{BC-ineq} is known as the Brown-Colbourn inequality \cite[Theorem 3.1]{Brown-Colbourn}.


\section{Rank-$2$ Matroids}

Let $M$ be a loopless matroid of rank $2$.  The independence complex of $M$ is a complete multipartite graph whose partite sets $E_1,\ldots,E_t$ are the parallelism classes of $M$.  Let $s_i:=|E_i|$.  Choose one representative $e_i \in E_i$ from each parallelism class of $M$ so that the simplification of $M$ is a complete graph on $\{e_1,\ldots,e_t\}$, and let $\widetilde{E}_i = E_i \setminus e_i$.  Clearly
\begin{eqnarray*}
f_0(M) &=& \sum_{i=1}^t(s_i-1) + t  \\
\text{and } f_1(M) &=& \sum_{1 \leq i<j \leq t}(s_i-1)(s_j-1) + (t-1)\sum_{i=1}^t(s_i-1) + {t \choose 2},
\end{eqnarray*}
and hence,
\begin{eqnarray*}
h_1(M) &=& \sum_{i=1}^t(s_i-1) + (t-2)  \\
\text{and } h_2(M) &=& \sum_{1 \leq i<j \leq t}(s_i-1)(s_j-1) + (t-2)\sum_{i=1}^t(s_i-1) + {t-1 \choose 2}.
\end{eqnarray*}

Consider the pure $O$-sequence $\mathcal{O}$ with 
\begin{eqnarray*}
\mathcal{O}_1 &=& \{x_1,\ldots,x_{t-2}\} \cup \{x_e: e \in \widetilde{E}_i, 1 \leq i \leq t\} \\
\mathcal{O}_2 &=& \{x_ex_{e'}: e \in \widetilde{E}_i, e' \in \widetilde{E}_j, 1 \leq i<j\leq t\} \\
&& \cup \{x_ix_e: e \in \widetilde{E}_j, 1 \leq i \leq t-2, 1 \leq j \leq t\} \\
&& \cup \{\text{degree $2$ monomials in } x_1,\ldots,x_{t-2}\}.
\end{eqnarray*}

We see that $h(M) = F(\mathcal{O})$, which proves the following theorem. 

\begin{theorem}
Let $M$ be a matroid of rank $2$.  Then $h(M)$ is a pure $O$-sequence. 
\end{theorem}


\section{Corank-$2$ Matroids}

In this section, we aim to prove Conjecture \ref{Stanley-conjecture} for corank-$2$ matroids.  

\begin{theorem}\label{corank2}
Let $M$ be a matroid of rank 2.  Then $h(M^*)$ is a pure $O$-sequence. 
\end{theorem}

\begin{proof}
As before, let $E_1,\ldots,E_t$ denote the parallelism classes of $M$.  Impose a total order on the ground set $E(M)$ so that $v_i < v_j$ for all $v_i \in E_k$ and $v_j \in E_{\ell}$ with $1 \leq k < \ell \leq t$.

For each basis $B = \{v_i,v_j\}$ of $M$ with $v_i \in E_k$, $v_j \in E_{\ell}$, and $k<\ell$, let 
\begin{eqnarray*}
a_1(B)&:=&\#\{i'>i: v_{i'} \in E_k \cup \cdots \cup E_{\ell-1}\} \\
\text{and } a_2(B) &:=& \#\{j'>j: v_{j'} \in E_{\ell} \cup \cdots \cup E_t\},
\end{eqnarray*}
and set $\mu_B:= x_1^{a_1(B)}x_2^{a_2(B)}$.  We claim that $\mathcal{O}:=\{\mu_B: B \in \B(M)\}$ is a pure order ideal and that $F(\mathcal{O}) = h(M^*)$.

\begin{figure}[ht]
\begin{center}
\epsfig{file=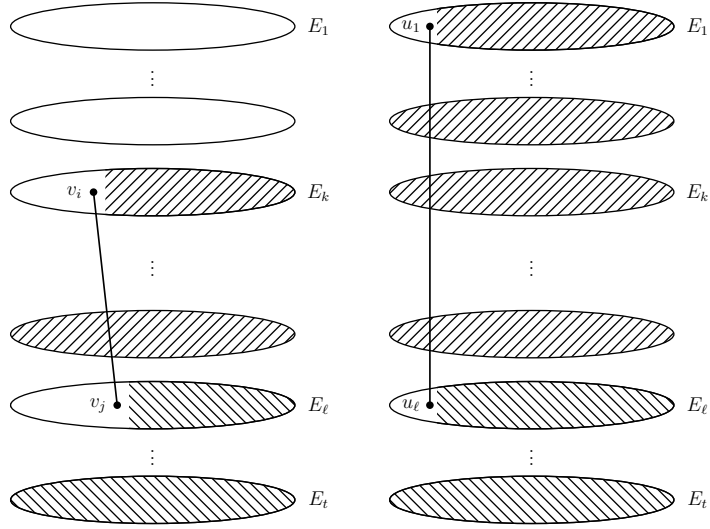, height = 7cm, bbllx=110, bblly=355, bburx=525, bbury=670}
\end{center}
\caption{The bases $B = \{v_i,v_j\}$ (left) and $\widetilde{B} = \{u_1,u_{\ell}\}$ (right) with their externally passive elements shaded.}
\label{corank2-picture}
\end{figure}

We see that $a_1(B)$ counts the number of elements $v \in E(M) \setminus B$ that are externally passive in $B$ for which $v_i < v < v_j$ (shown in Figure \ref{corank2-picture} (left) shaded with lines of slope $1$); and $a_2(B)$ counts the number of elements $v \in E(M) \setminus B$ that are externally passive in $B$ for which $v_j < v \leq v_n$ (shown in Figure \ref{corank2-picture} (left) shaded with lines of slope $-1$).  Since $a_1(B)+a_2(B)$ counts the number of externally passive elements in $B$, Equation \eqref{h-nums-dual} shows that $h(M^*) = F(\mathcal{O})$. 

To see that $\mathcal{O}$ is an order ideal, we need only show that if $\nu|\mu_B$ and $\deg(\nu) = \deg(\mu_B)-1$, then $\nu \in \mathcal{O}$.  Let $B = \{v_i,v_j\}$ as before.  If $a_1(B) > 0$, consider $B' = \{v_{i+1},v_j\} \in \I(M)$.  Clearly $a_1(B') = a_1(B)-1$ and $a_2(B') = a_2(B)$ so that $\mu_{B'} \in \mathcal{O}$ and $\deg(\mu_{B'}) = \deg(\mu_B)-1$.  If $a_2(B)>0$, we must consider two possible cases.  If $v_{j+1} \in E_{\ell}$, then consider $B'' = \{v_i,v_{j+1}\} \in \I(M)$.  Again $a_1(B'') = a_1(B)$ and $a_2(B'') = a_2(B)-1$ so that $\mu_{B''} = x_1^{a_1(B)}x_2^{a_2(B)-1}$.  On the other hand, if $v_{j+1} \in E_{\ell+1}$, then $v_{j-a_1(B)} \in E_{k'}$ for some $k' \leq \ell$, and so $B''' = \{v_{j-a_1(B)},v_{j+1}\} \in \I(M)$.  Again we see that $\mu_{B'''} = x_1^{a_1(B)}x_2^{a_2(B)-1}.$  This establishes that $\mathcal{O}$ is an order ideal.  

Finally, we must show that $\mathcal{O}$ is pure.  For each $1 \leq i \leq t$, let $u_i$ denote the smallest element of $E_i$.  For any basis $B = \{v_i,v_j\}$ as above, let $\widetilde{B} = \{u_1,u_{\ell}\}$.  As Figure \ref{corank2-picture} (right) indicates, $a_1(B) \leq a_1(\widetilde{B})$ and $a_2(B) \leq a_2(\widetilde{B})$, and hence $\mu_B|\mu_{\widetilde{B}}$.  Moreover, $\deg(\mu_{\widetilde{B}}) = |E_1| + \cdots + |E_t|-2$, and hence each such monomial $\mu_{\widetilde{B}}$ has the same degree. 
\end{proof}

The techniques used to prove Theorem \ref{corank2} can be easily extended to prove that $h(M^*)$ is a pure $\mathcal{O}$-sequence for any matroid $M$ whose simplification is a uniform matroid.  The reader may easily check, however, that these techniques may not be used to prove Stanley's conjecture when $M$ is the Fano matroid, thus these techniques may not be extended to corank 3.


\section{Rank-3 Matroids}

Our goal for this section is to give a simple, short, geometric-combinatorial
proof of the following theorem, first proved in
\cite{Ha-Stokes-Zanello} for the case that $d=3$ using the language of
commutative algebra.

\begin{theorem} \label{rank3-o-seq}
Let $M$ be a loopless matroid of rank $d \geq 3$.  The vector $(1,h_1(M),h_2(M),h_3(M))$ is a pure $O$-sequence. 
\end{theorem}

\begin{lemma} \label{h-vec-upper}
For any positive integers $s_1,\ldots,s_t$, the vector $\mathbf{h} = (1,h_1,h_2,h_3)$ with 
\begin{eqnarray*}
h_1 &=& \sum_{i=1}^t (s_i-1)+(t-d), \\
h_2 &=& \sum_{1 \leq i < j \leq t} (s_i-1)(s_j-1) + (t-d)\sum_{i=1}^t(s_i-1) + {t-d+1 \choose 2},\\
h_3 &=& \sum_{1 \leq i<j<k \leq t}(s_i-1)(s_j-1)(s_k-1) + (t-d)\sum_{1 \leq i<j \leq t}(s_i-1)(s_j-1) \\
&&+ {t-d+1 \choose 2}\sum_{i=1}^t(s_i-1)+{t-d+2 \choose 3},
\end{eqnarray*}
is a pure $O$-sequence.
\end{lemma}
\begin{proof}
Consider disjoint sets $\widetilde{E}_1,\ldots,\widetilde{E}_t$ with $|\widetilde{E}_i| = s_i-1$ for all $i$.  We will construct a pure order ideal $\mathcal{O}$ with $F(\mathcal{O}) = \mathbf{h}$ whose degree-one terms are $$\mathcal{O}_1 = \{x_1,\ldots,x_{t-d}\} \cup \{x_e: e \in \widetilde{E}_i\}_{i=1}^t.$$ We explicitly construct such an order ideal by setting 
\begin{eqnarray*}
\mathcal{O}_2 &=& \{x_ex_{e'}: e \in \widetilde{E}_i, e' \in \widetilde{E}_j, 1 \leq i <j \leq t\} \\
&& \cup \{x_jx_e: e \in \widetilde{E}_i, 1 \leq i \leq t, 1 \leq j \leq t-d\} \\
&& \cup \{\text{all degree 2 monomials in } x_1,\ldots,x_{t-d}\}
\end{eqnarray*}
and
\begin{eqnarray*}
\mathcal{O}_3 &=& \{x_ex_{e'}x_{e''}: e \in \widetilde{E}_i, e' \in \widetilde{E}_j, e'' \in \widetilde{E}_k, 1 \leq i<j<k \leq t\} \\
&& \cup \{x_kx_ex_{e'}: e \in \widetilde{E}_i, e' \in \widetilde{E}_j, 1 \leq k \leq t-d, 1 \leq i<j \leq t\} \\
&& \cup \{x_jx_kx_e: e \in \widetilde{E}_i, 1 \leq j<k \leq t-d, 1 \leq i \leq t\} \\
&& \cup \{x_j^2x_e: e \in \widetilde{E}_i, 1 \leq i \leq t, 1 \leq j \leq t-d\} \\
&& \cup \{\text{all degree 3 monomials in } x_1,\ldots,x_{t-d}\}.
\end{eqnarray*}
\end{proof}

\begin{lemma} \label{h-vec-lower}
For any positive integers $s_1,\ldots,s_t$, the vector $\mathbf{h}' = (1,h_1,h_2,h_3)$ with 
\begin{eqnarray*}
h_1 &=& \sum_{i=1}^t (s_i-1)+(t-d), \\
h_2 &=& \sum_{1 \leq i < j \leq t} (s_i-1)(s_j-1) + (t-d)\sum_{i=1}^t(s_i-1) + {t-d+1 \choose 2},\\
h_3 &=& \sum_{1 \leq i<j \leq t} (s_i-1)(s_j-1) + (t-d-1)\sum_{i=1}^t(s_i-1) + {t-d \choose 2}+1,
\end{eqnarray*}
is a pure $O$-sequence.
\end{lemma}

\begin{proof}
As in the proof of Lemma \ref{h-vec-upper}, let $\widetilde{E}_1,\ldots,\widetilde{E}_t$ be disjoint sets with $|\widetilde{E}_i| = s_i-1$. Recall the order ideal $\mathcal{O}$ constructed in the proof of Lemma \ref{h-vec-upper}.  We will construct a pure order ideal $\widetilde{\mathcal{O}}$ with $F(\widetilde{\mathcal{O}}) = \mathbf{h}'$ such that $\widetilde{\mathcal{O}}_1 = \mathcal{O}_1$, $\widetilde{\mathcal{O}}_2 = \mathcal{O}_2$, and $\widetilde{\mathcal{O}}_3 \subseteq \mathcal{O}_3$.  We set 
\begin{eqnarray*}
\widetilde{\mathcal{O}}_3 &=& \{x_1x_ex_e': e \in \widetilde{E}_i, e' \in \widetilde{E}_j, 1 \leq i<j\leq t\} \\
&& \cup \{x_j^2x_e: e \in \widetilde{E}_i, 1 \leq i \leq t, 2 \leq j \leq t-d\} \\
&& \cup \{x_i^2x_j: 1 \leq i<j \leq t-d\} \cup \{\mu_0\},
\end{eqnarray*}
where $\mu_0$ is a monomial defined as follows: if $\widetilde{E}_1 \cup \cdots \cup \widetilde{E}_t$ is nonempty, choose some $e_0 \in \widetilde{E}_1\cup \cdots \cup \widetilde{E}_t$ and set $\mu_0 = x_1^2x_{e_0}$.  Otherwise, set $\mu_0 = x_1^3$.  This distinction in the monomial $\mu_0$ is necessary for handling the cases in which $|\widetilde{E}_1 \cup \cdots \cup \widetilde{E}_t| \leq 1$.
\end{proof}

\textit{Proof: (Theorem \ref{rank3-o-seq})}

Let $E_1,\ldots,E_t \subseteq E(M)$ denote the parallelism classes of $M$, and set $s_i:=|E_i|$.  Choose one representative $e_i$ from each class $E_i$, and let $W = \{e_1,\ldots,e_t\}$.  Observe that $\Delta:= M|_W$ is a simple matroid of rank $d$.  Let $\widetilde{E}_i = E_i \setminus \{e_i\}$, and notice that for any choice of $\widetilde{e}_{i_j} \in E_{i_j}, \{\widetilde{e}_{i_1},\ldots,\widetilde{e}_{i_k}\} \in \mathcal{I}(M)$ if and only if $\{e_{i_1}, \ldots, e_{i_k}\} \in \Delta$.  Thus
\begin{eqnarray*}
f_0(M) &=& \sum_{i=1}^ts_i \text{ and hence}\\
h_1(M) &=& \sum_{i=1}^t(s_i-1) + (t-d); \\
f_1(M) &=& \sum_{1 \leq i<j \leq t}s_is_j  \\
&=& \sum_{1 \leq i<j \leq t}(s_i-1)(s_j-1) + (t-1)\sum_{i=1}^t(s_i-1) + {t \choose 2} \text{ and hence} \\
h_2(M) &=& \sum_{1 \leq i<j \leq t}(s_i-1)(s_j-1) + (t-d)\sum_{i=1}^t(s_i-1) + {t-d+1 \choose 2}; \\
f_2(M) &\leq& \sum_{1 \leq i < j < k \leq t} s_is_js_k \text{ and hence} \\
h_3(M) &\leq& \sum_{1 \leq i<j<k \leq t}(s_i-1)(s_j-1)(s_k-1) + (t-d) \sum_{1 \leq i < j \leq t}(s_i-1)(s_j-1) \\
&& + {t-d+1 \choose 2}\sum_{i=1}^t(s_i-1) + {t-d+2 \choose 3}.
\end{eqnarray*}
On the other hand,  by the Brown-Colbourn inequality \eqref{BC-ineq}, 
\begin{eqnarray*}
h_3(M) &\geq& h_2(M) - h_1(M) + h_0(M) \\
&=& \sum_{1 \leq i<j \leq t} (s_i-1)(s_j-1) + (t-d-1)\sum_{i=1}^t(s_i-1) + {t-d \choose 2}+1.
\end{eqnarray*}

We construct a pure order ideal $\mathcal{O}'$ with $F(\mathcal{O}') = h(M)$ as follows.  Following the notation used in Lemmas \ref{h-vec-upper} and \ref{h-vec-lower}, we set $\mathcal{O}'_1 = \mathcal{O}_1$; $\mathcal{O}'_2 = \mathcal{O}_2$, and choose $\widetilde{\mathcal{O}}_3 \subseteq \mathcal{O}'_3 \subseteq \mathcal{O}_3$ with $|\mathcal{O}'_3| = h_3(M)$.  

\hfill{$\Box$}


\section{Matroids on at most 9 elements}

This part of our paper is experimental and is crucially based on the data provided to us by Dillon Mayhew and Gordon Royle. They constructed a computer database of all 385,369 matroids on at most nine elements \cite{Mayhew-Royle}. We used this data to generate a list of all possible $h$-vectors of matroid complexes on at most nine elements.  Given a loopless, coloopless matroid $M$ of rank $d$ on $n$ elements, we searched for a pure $O$-sequence $\mathcal{O}$ with $h(M) = F(\mathcal{O})$ in the following way: we know that $h_d(M)$ counts the number of top-degree monomials in $\mathcal{O}$, and $h_1(M) = n-d$ counts the number of variables (degree-one terms) in $\mathcal{O}$.  By sampling the space of monomials of degree $h_d(M)$ on $h_1(M)$ variables, we can generate thousands of pure $O$-sequences that are candidates to be $h$-vectors of matroid complexes. Of course, because of the tremendous restrictions that the basis exchange axioms place on matroids, and hence also on their $h$-vectors, we often generated pure $O$-sequences that were not matroid $h$-vectors.  For example $(1,5,15,27,22)$ and $(1,5,15,27,35)$ are both valid pure $O$-sequences which were generated, but the only $h$-vectors of matroid complexes of rank 4 with initial value $(1,5,15,27,*)$ are

\begin{verbatim}
(1 5 15 27 0)  (1 5 15 27 19) (1 5 15 27 20) (1 5 15 27 21) (1 5 15 27 24) 
(1 5 15 27 25) (1 5 15 27 26) (1 5 15 27 27) (1 5 15 27 30) (1 5 15 27 36).
\end{verbatim}

To generate the $O$-sequences, we used a combination of {\tt Perl} and {\tt Maple} code available at \url{www.math.ucdavis.edu/~ykemper/matroids.html}. 
The key idea is that  $m=h_d(M)$ provides us with the size of a monomial set to be sampled in a given number of variables $k=h_1(M)$. 
Specifically, we started with an initial set of $m$ monomials within the simplex $\{ (x_1,x_2,\dots,x_k) : \sum_i x_i=d, \ x_i \geq 0 \}$, then calculated the corresponding pure $O$-sequence by counting the number of monomials of each degree less than or equal to $d$ which divide one or more of the initial monomials.  One approach we used to generate large numbers of $O$-sequences was to sample randomly within the lattice points of this simplex.  Another was to perform ``mutation'' operations based on the idea that within the simplex, all 
lattice points are connected by the vectors $e_i-e_j$ of the root system $A_n$. We could therefore move ``locally'' from one pure order ideal to the next.  In addition, we partially adapted a simulated annealing
type method to search for particular $h$-vectors (program labeled {\tt Boxy}) not found in our random sampling.  {\tt Boxy} is also quite useful for computing the $O$-sequence of a family of 
monomials given the top-degree monomials of that family.  For example, by entering  $[[0,0,0,5], [0,0,2,3], [1,3,0,1]]$, one can obtain the 
corresponding $O$-sequence $(1,4,7,7,6,3)$.

The data we present in the web site is grouped by rank and corank. The largest groups are concentrated around rank four and corank five. 
We have decided not to include the cases of rank one, two, and three, and corank one and two
because they are consequences of theorems presented earlier. Note that we have not listed monomials for matroids with coloops: a matroid having $j$ coloops has an $h$-vector with $j$ zeros at the end, and the non-zero entries correspond to the $h$-vector of the same matroid with all coloops contracted. Since this new matroid also has a ground set of at most nine elements, a family of monomials has been provided for it elsewhere in the table.  The total number distinct matroid $h$-vectors (including $h$-vectors corresponding to matroids with coloops) and the total number of matroids per rank and corank are listed below.  When the rank plus corank is greater than nine, we have no information on the quantities of matroids or distinct $h$-vectors, and have indicated this with `--.' \\

\begin{table}[hbt]
\begin{tabular}{|c||p{1cm}|p{1.05cm}|p{1.1cm}|p{1.2cm}|p{1.2cm}|p{1.1cm}|p{1.1cm}|p{1.05cm}|p{1cm}|p{1cm}|}
\hline
\textbf{Rank}/\textbf{Corank} & \textbf{0} & \textbf{1} & \textbf{2} & \textbf{3} & \textbf{4} & \textbf{5} & \textbf{6} & \textbf{7} & \textbf{8} & \textbf{9} \\
\hline
\hline
\textbf{0} & 0 & 1 & 1 & 1 & 1 & 1 & 1 & 1 & 1 & 1\\
\hline
\textbf{1} & 1 & 1 & 1 & 1 & 1 & 1 & 1 & 1 & 1 & --\\
\hline
\textbf{2} & 1 & 2 & 4 & 6 & 8 & 12 & 17 & 20 & -- & --\\
\hline
\textbf{3} & 1 & 3 & 9 & 22 & 49 & 101 & 196 & -- & -- & --\\
\hline
\textbf{4} & 1 & 4 & 18 & 67 & 244 & 816 & -- & -- & -- & --\\
\hline
\textbf{5} & 1 & 5 & 31 & 186 & 1132 & -- & -- & -- & -- & --\\
\hline
\textbf{6} & 1 & 6 & 51 & 489 & -- & -- & -- & -- & -- & -- \\
\hline
\textbf{7} & 1 & 7 & 79 & -- & -- & -- & -- & -- & -- & --\\
\hline
\textbf{8} & 1 & 8 & -- & -- & -- & -- & -- & -- & -- & --\\
\hline
\textbf{9} & 1 & -- & -- & -- & -- & -- & -- & -- & -- & --\\
\hline
\end{tabular}
\begin{center}
\caption{Number of distinct matroid $h$-vectors for particular rank and corank}
\end{center}
\end{table}

\begin{table}[h]
\begin{tabular}{|c||p{1cm}|p{1.05cm}|p{1.1cm}|p{1.2cm}|p{1.2cm}|p{1.1cm}|p{1.1cm}|p{1.05cm}|p{1cm}|p{1cm}|}
\hline
\textbf{Rank}/\textbf{Corank} & \textbf{0} & \textbf{1} & \textbf{2} & \textbf{3} & \textbf{4} & \textbf{5} & \textbf{6} & \textbf{7} & \textbf{8} & \textbf{9} \\
\hline
\hline
\textbf{0} & 0 & 1 & 1 & 1 & 1 & 1 & 1 & 1 & 1 & 1\\
\hline
\textbf{1} & 9 & 8 & 7 & 6 & 5 & 4 & 3 & 2 & 1 & --\\
\hline
\textbf{2} & 8 & 14 & 24 & 30 & 40 & 42 & 42 & 29 & -- & --\\
\hline
\textbf{3} & 7 & 18 & 45 & 100 & 210 & 434 & 950 & -- & -- & --\\
\hline
\textbf{4} & 6 & 20 & 72 & 255 & 1664 & 189274 & -- & -- & -- & --\\
\hline
\textbf{5} & 5 & 20 & 93 & 576 & 189889 & -- & -- & -- & -- & --\\
\hline
\textbf{6} & 4 & 18 & 102 & 1217 & -- & -- & -- & -- & -- & -- \\
\hline
\textbf{7} & 3 & 14 & 79 & -- & -- & -- & -- & -- & -- & --\\
\hline
\textbf{8} & 2 & 8 & -- & -- & -- & -- & -- & -- & -- & --\\
\hline
\textbf{9} & 1 & -- & -- & -- & -- & -- & -- & -- & -- & --\\
\hline
\end{tabular}
\begin{center}
\caption{Total number of matroids, for particular rank and corank}
\end{center}
\end{table}


\section*{Acknowledgements}
The first author was partially supported by NSF grant DMS-0914107, the second and third authors were supported by NSF VIGRE grant DMS-0636297.
We are truly grateful to Dillon Mayhew and Gordon Royle for the opportunity to use their data in our investigations. We are also grateful to
David Haws for his help on getting this project started. We are incredibly grateful to Jonathan Browder and Ed Swartz for a number of thoughtful 
and insightful conversations.  Criel Merino was kind enough to give us several useful comments and references.



\bibliography{matroids}
\bibliographystyle{plain}

\end{document}